\documentclass{amsart}
\usepackage{graphicx}
\usepackage{latexsym}
\usepackage{color}
\usepackage{amsfonts}
\usepackage[all]{xy}
\usepackage{amssymb, mathrsfs, amsfonts, amsmath}
\usepackage{amsbsy}
\usepackage{amsfonts}
\newtheorem{corollary}{Corollary}

\newtheorem{lemma}{Lemma}
\newtheorem{proposition}{Proposition}
\newtheorem{remark}{Remark}
\newtheorem{theorem}{Theorem}
\newtheorem{example}{Example}
\numberwithin{equation}{section}

\newcommand{\be}{\begin{equation}}
	\newcommand{\ee}{\end{equation}}
\newcommand{\ben}{\begin{enumerate}}
	\newcommand{\een}{\end{enumerate}}
\newcommand{\beq}{\begin{eqnarray}}
	\newcommand{\eeq}{\end{eqnarray}}
\newcommand{\beqn}{\begin{eqnarray*}}
	\newcommand{\eeqn}{\end{eqnarray*}}

\usepackage[pagebackref]{hyperref}

\begin{document}
	\title{On cylindrical symmetric Finsler metrics with vanishing Douglas
curvature}

\author{Newton Sol\'orzano}
\address[N. Sol\'orzano]{ILACVN - CICN, Universidade Federal da Integra\c c\~ao Latino-Americana, UNILA}
\curraddr{Itaipu Parquetec, Foz do Igua\c cu-PR, 85867-970 - Brasil}	\email{nmayer159@gmail.com}

\author{Dik D. Lujerio Garcia}
\address[D. Lujerio]{Departamento Académico de Matemática de la Facultad de Ciencias, Universidad Nacional Santiago Ant\'unez de Mayolo, UNASAM}
\curraddr{Jr. Augusto B.Leguia N° 110, Huaraz - Per\'u}	\email{dlujeriog@unasam.edu.pe}	
 
 \author{V\'ictor Le\'on}
\address[V. Le\'on]{ILACVN - CICN, Universidade Federal da Integra\c c\~ao Latino-Americana, UNILA}
\curraddr{Parque tecnol\'ogico de Itaipu, Foz do Igua\c cu-PR, 85867-970 - Brasil}
\email{victor.leon@unila.edu.br}

\author{Alexis Rodr\'iguez Carranza}
	
\address[A. Rodr\'iguez Carranza]{Departamento de Ciencias, Universidad Privada del Norte, UPN}
\curraddr{Sede San Isidro, Av. El Ejercito 920, Trujillo - Per\'u}	\email{alexis.rodriguez@upn.edu.pe}	
	\begin{abstract}
 In this paper, we consider the {\em cylindrically symmetric Finsler metrics} and we obtain their Douglas curvature. Furthermore, we obtain the differential equation system of the cylindrically symmetric Finsler metrics with vanishing Douglas curvature. Many examples are included.
	\end{abstract}
	
	\date{\today}

	\maketitle
\section{Introduction}

The Douglas curvature, introduced by J. Douglas \cite{D} in 1927, is an important projective invariant in Finsler geometry. That is, if two Finsler metrics $F$ and $\bar{F}$ are projectively equivalent, then $F$ and $\bar{F}$ have the same Douglas curvature. A Finsler metric is called \textit{Douglas metric} if their Douglas curvature vanishes. Douglas metrics are rich, in the sense that every Riemman metrics and projectively flat metrics are also Douglas metrics. Besides there are a lot of examples of Douglas metrics that are not Riemanniana nor projectively flat. For instante, a Randers metric $F=\alpha + \beta$ is a Douglas metric if and only if $\beta$ is closed (\cite{Bacso1997}).

On the other hand, there exist important Finsler metrics in the literature  which satisfy 
 \begin{align}\label{eq:1}F\left((x^0,O\overline{x}),(y^0,O\overline{y})\right)=F\left((x^0,\overline{x}),(y^0,\overline{y})\right), \text{ for every } O\in O(n),
 \end{align}
  where 
 $ x=(x^0,\overline{x})=(x^0,x^1,\ldots,x^n)\in M=I\times \mathbb{R}^n, y=(y^0,\overline{y})=(y^0,y^1,\ldots,y^n)\in T_xM,$ like the Shen’s fish tank metric on $\Omega= \mathbb{B}^2\times \mathbb{R}\subset  \mathbb{R}^3 $: 
\begin{align*}
F=\frac{\sqrt{(-x^2y^1+x^1y^2)^2+((y^1)^2+(y^2)^2+(y^3)^2)(1-(x^1)^2-(x^2)^2)}}{1-(x^1)^2-(x^2)^2} - \frac{x^2y^1-x^1y^2}{1-(x^1)^2-(x^2)^2}, 
\end{align*}
where $ x=(x^1,x^2,x^3)\in \mathbb{B}^2\times \mathbb{R} $ and $ y=(y^1,y^2,y^3)\in T_xO, $
or, the spherically symmetric (or orthogonal invariance) Finsler metric \cite{HM1,Z} : \[ F=\vert y\vert \phi\left(\vert x\vert, \frac{ \langle x,y\rangle }{\vert y \vert}\right), \]
where $ x\in M=\mathbb{R}^{n+1}, y\in T_xM, $
or the warped metrics \cite{Zhao2018,Kozma2001,Liu2019,marcal2023} defined on $ I\times \mathbb{R}^n $ of the form 
\begin{align*}	F&=\vert \overline{y}\vert\phi\left(x^0,\frac{y^0}{\vert\overline{y}\vert}\right), &
	F&=\vert \overline{y}\vert \phi\left(\frac{y^0}{\vert\overline{y}\vert},\vert \overline{x}\vert\right).
\end{align*}
 A Finsler metric $F$ is called {\em cylindrically symmetric } (or weakly orthogonally invariant in an alternative terminology in \cite{Liu2024}) if $F$ satisfies \eqref{eq:1}.  In \cite{Liu2024}, the authors showed that cylindrically symmetric metrics are non-trivial in the sense that this type of metric is not of orthogonal invariance (see Proposition 2.2 in \cite{Liu2024}).

In \cite{Solorzano2022} the author showed that every cylindrically symmetric Finsler metric can be written as 
\[F(x,y)=\vert \overline{y}\vert{\phi\left(x^0,\vert \overline{x}\vert,\frac{\langle\overline{x},\overline{y}\rangle}{\vert\overline{y}\vert},\frac{y^0}{\vert \overline{y}\vert}\right)},\]
where $|\cdot|$ and $\langle \cdot,\cdot\rangle$ are, respectively, the 
standard Euclidean norm and inner product on $\mathbb{R}^n$. Furthermore, in \cite{Solorzano2023} the authors provide necessary and sufficient conditions for $F=\vert \overline{y}\vert \phi$ to be a Finsler metric (Theorem 1 in \cite{Solorzano2023}). 

In Section 2 we give some preliminaries and recall some recent results about cylindrically symmetric Finsler metrics. In section 3 we study their Douglas curvature. Specifically, we obtain the Douglas curvature (see Theorem \ref{Dcurv}) and the characterization of the vanishing Douglas curvature (see Theorem \ref{maintheo1}). In Section 4 we give some examples.

\section{Preliminaries}

In this section, we give some notations, definitions, and lemmas that will be used in the proof of our main results.
Let $M$ be a manifold, and let $TM=\cup_{x\in M}T_xM$ be the tangent
bundle of $M$, where $T_xM$ is the tangent space at $x\in M$. We
set $TM_o:=TM\setminus\{0\}$ where $\{0\}$ stands for
$\left\{(x,\,0)|\, x\in M,\, 0\in T_xM\right\}$. A {\em Finsler
metric} on $M$ is a function $F:TM\to [0,\,\infty)$ with the
following properties:
\begin{itemize}
    \item[(a)] $F$ is $C^{\infty}$ on $TM_o$;

\item[(b)] At each point $x\in M$, the restriction $F_x:=F|_{T_xM}$ is a
Minkowski norm on $T_xM$.
 
\end{itemize}
Let  $\mathbb{B}^n(\rho)\subset\mathbb{R}^n$ the $n$ dimensional open ball of radius $\rho$ and centered at the origin ($n\geq 2$). Set 
 $M=I\times \mathbb{B}^n(\rho)\subset \mathbb{R}\times\mathbb{R}^n,$ with coordinates on $ TM $
\begin{align}
	x&=(x^0, \overline{x}), \quad \overline{x}=(x^1,\ldots,x^n),\label{coordx}\\
	y&=(y^0, \overline{y}), \quad \overline{y}=(y^1,\ldots,y^n).\label{coordy}
\end{align} 
Throughout our work, the following convention for indices is adopted: 
\begin{align*}
	0\leq&A, B, \ldots \leq n;\\
	1\leq&i,j,\ldots \leq n.
\end{align*} 
Introducing the notation 
\begin{align}\label{rs} 
	 r&:=|\overline{x}|,  &s&:=\frac{\langle \overline{x},\,\overline{y}\rangle}{\vert\overline{y}\vert},&  z&:= \frac{y^0}{\vert\overline{y}\vert},
\end{align} 
where $|\cdot|$ and $\langle\cdot,\cdot\rangle$ are, respectively, the 
standard Euclidean norm and inner product on $\mathbb{R}^n$. 

In \cite{Solorzano2022}, the authors proved that, if the Finsler metric $F$ satisfies \eqref{eq:1}, then there exist a positive function $\phi:\mathbb{R}^4 \to\mathbb{R} $ such that,
\begin{align}\label{def:F}
F(x,y) = \vert \overline{y}\vert \phi(x^0,r,s,z).
\end{align}
On the other hand, defining $ \Omega $ and $ \Lambda $ as,
\begin{align}
	\Omega:=&\phi-s\phi_s-z\phi_z \label{DefOmega},\\
	 \Lambda:=& \Omega \phi_{zz}+(r^2-s^2)(\phi_{ss}\phi_{zz}-\phi^2_{sz}),\label{Def:Lambda}
\end{align}
where, the sub-index $s,z$ are the partial derivatives respect to $s$ and $z$ respectively, the Hessian matrix $ \left(g_{AB}\right)=\frac{1}{2} [F^2]_{y^Ay^B}=
\left(
\begin{array}{c|c}
	g_{00} & g_{0j} \\
	\hline
	g_{i0} & g_{ij}
\end{array}
\right),$ is given by

\begin{align*}
	g_{00}=&\phi^2_z+\phi\phi_{zz},\\
	g_{i0}=&g_{0i}=(\phi\Omega)_zu^i+(\phi_s\phi_z+\phi\phi_{sz})x^i,\\
	g_{ij}=&\phi\Omega\delta_{ij} + X_{ij},
\end{align*}
where $ X_{ij}=(u^i, x^i)\left(\begin{array}{c c}
	-(s(\phi\Omega)_s+z(\phi\Omega)_z) & (\phi\Omega)_s\\
	(\phi\Omega)_s& (\phi_s^2+\phi\phi_{ss}) 
\end{array}\right)\left(\begin{array}{c}
	u^j\\
	x^j
\end{array}\right),$ with $u^j=\dfrac{y^j}{|\overline{y}|}$.  
 
Note that, the determinant of $ g_{AB} $ is given by \begin{align*}
	\det(g_{AB})=\phi^{n+2}\Omega^{n-2}\Lambda.
\end{align*}
With this, we recall the next result about the necessary and sufficiency condition for the function $F=\vert \overline{y}\vert \phi(x^0,r,s,z)$ to be a Finsler metric \cite{Solorzano2023}.
\begin{proposition}\normalfont
		Let $F=\vert\overline{y}\vert{\phi(x^0,r,s,z)}$ be a Finsler metric defined on $ M $, where $ z=\frac{y^0}{\vert\overline{y}\vert}, $ $r=\vert\overline{x}\vert$, $ s=\frac{\langle\overline{x},\overline{y}\rangle}{\vert\overline{y}\vert} $ and  $TM $  with coordinates   \eqref{coordx}-\eqref{coordy}. Then $  F$ is a Finsler metric if, and only if, the positive function $ \phi $ satisfies  $ \Lambda>0 $ for $ n=2 $  with additional inequality, $ \Omega>0 $ for $ n\geq 3. $
\end{proposition}
The next proposition  gives us one the most important quantities in Finsler Geometry: The geodesic coefficients \begin{align*}
	 G^A=Py^A+Q^A, \end{align*}
where
\begin{align*}
	P:=&\frac{F_{x^C}y^C}{2F}, & Q^A:=\frac{F}{2}g^{AB}\left\{F_{x^Cy^B}y^C-F_{x^B}\right\},
\end{align*}
where $g^{AB}$ is the inverse  of the matrix $g_{AB}$ (see details in  \cite{Solorzano2023}).

\begin{proposition}\label{prop2}\normalfont Let
	$F=\vert\overline{y}\vert{\phi(x^0,r,s,z)}$ be a Finsler metric defined on $ M $, where $ z=\frac{y^0}{\vert\overline{y}\vert}, $ $r=\vert\overline{x}\vert$, $ s=\frac{\langle\overline{x},\overline{y}\rangle}{\vert\overline{y}\vert} $ and  $TM $  with coordinates   \eqref{coordx}-\eqref{coordy}.  Then the geodesic spray coefficients $ G^A $ are given by
	\begin{align}
        	G^0&=u^2\left\{z(W+sU)+L\right\},\label{eq:G^0}\\
		G^i&=u^2Wu_i + u^2Ux^i,\label{eq:G^i}
	\end{align}
where $ u=\vert\overline{y} \vert$, $u_i=\frac{y^i}{u}$, $ \Omega, \Lambda $ are given in \eqref{DefOmega}, \eqref{Def:Lambda} respectively, and
\begin{align}
	W&:=\frac{1}{\phi}\left\{\frac{\varphi}{2}-s\phi U - \phi_zL - (r^2-s^2)\phi_sU\right\},\nonumber\\
L&:=\frac{\Omega}{2 \Lambda}(\varphi_z-2\phi_{x^0})-(r^2-s^2)V,\label{def:A}\\
 U&:=\frac{1}{2\Lambda}\left\{\left(\varphi_s-\frac{2}{r}\phi_r\right)\phi_{zz}-\left(\varphi_z-2\phi_{x^0}\right)\phi_{sz}\right\},\label{def:U}\\
	V&:=\frac{1}{2\Lambda}\left\{\left(\varphi_s-\frac{2}{r}\phi_r\right)\phi_{sz}-\left(\varphi_z-2\phi_{x^0}\right)\phi_{ss}\right\},\nonumber\\
	\varphi &:=z\phi_{x^0}+\frac{s}{r}\phi_r+\phi_s.\nonumber
\end{align}
\end{proposition}

\section{Douglas curvature}

A Finsler metric on a $n-$diensional manifold $N$ is called a {\em Douglas metric} if its geodesic coefficients $G^i=G^i(x,\,y)$ are given in the following form
$$
G^i=\frac 12\Gamma^i_{jk}(x)y^jy^k+P(x,\,y)y^i,
$$
where $\Gamma^i_{jk}(x)$ are functions on $N$, in local coordinates,  and $P(x,\,y)$
is a local positively $y$-homogeneous function of degree one. 

In \cite{D}, Douglas introduced the local functions 
 $D_j{}^i{}_{kl}$ on ${T}N^n$ defined by
\[
D_j{}^i{}_{kl}:=\frac{\partial^3}{\partial y^j\partial y^k\partial
y^l}\left(G^i-\frac 1{n+1}\sum_m \frac{\partial G^m}{\partial
y^m}y^i\right),\]  in local coordinates $x^1,\ldots,x^n$ and $y=\sum_i y^i \partial/\partial x^i$. 
These functions are called  {\em Douglas curvature} \cite{D} and a Finsler metric 
 $F$ with $D_j{}^i{}_{kl}=0$ is called {\em Douglas metric}.

Before to obtain the Douglas curvature for a cylindrically symmetric Finsler metric $F=\vert \overline{y}\vert\phi(x^0,r,s,z)$ we claim the next.
\begin{lemma}\label{lem1}\normalfont Under the assumptions of the Proposition \ref{prop2}, we have the following equalities:
    \begin{align*}
	G^0-\frac{y^0}{n+2}\frac{\partial G^A}{\partial y^A}=&u^2R ,\\
	G^i-\frac{y^i}{n+2}\frac{\partial G^A}{\partial y^A}=&u^2Ux^i-u^2Tu_i.
\end{align*}
where
\begin{align}
	R=& \left\{L-\frac{z}{n+2}\left[L_z-(n-1)sU+(r^2-s^2)U_s\right]\right\},\label{def:R}\\
 	T=&\frac{1}{n+2}\left\{3sU+L_z+(r^2-s^2)U_s\right\}.\label{def:T}
\end{align}
\end{lemma} 
\begin{proof}
    From \eqref{rs}, we have the partial derivatives of  $ u=\vert \overline{y}\vert,  $ $ s $ and $ z, $  respect to $ y^i $  
\begin{align}
	u_j&=\frac{y^j}{u},\label{eq:u_j}\\
		u_{jk}&=\frac{1}{u}\left(\delta_{jk}-u_ju_k\right),\nonumber\\
	s_j&=\frac{1}{u}\left(x^j-su_j\right),\label{eq:s_j}\\
	z_l&=-\frac{z}{u}u_l. \label{eq:z_l}
	\end{align}
From \eqref{eq:u_j}, \eqref{eq:s_j} and \eqref{eq:z_l}, 
\begin{align}
    u_iu_i&=1, & u_ix^i&=s,\nonumber\\
    s_ix^i&=\frac{r^2-s^2}{u}, & s_iu_i&=0, \label{eq:us_ix^i}\\
    uz_ix^i&=-sz, & z_iu_i&=-\frac{z}{u},\label{eq:uz_ix^i}\\
    s_is_i&=\frac{r^2-s^2}{u^2}.\nonumber
\end{align}
Additionally, from \eqref{eq:G^0}, using \eqref{eq:us_ix^i} and \eqref{eq:uz_ix^i}, we have,

\begin{align}\label{eq:partialG^0}
	\frac{\partial G^0}{\partial y^0}&=u\left\{(W+sU) + z(W_z+sU_z) + L_z\right\}.
\end{align}
Note that $ uW $ (in \eqref{eq:G^i}) is positive homogeneous of degree 1 on $y=(y^0,\overline{y})$.  From Euler's theorem for homogeneous functions,
	\begin{align*}
		\frac{\partial uW}{\partial y^i}y^i=u(W-zW_z),
	\end{align*} 
then,
\begin{align}\label{eq:partialG^i}
	\sum \frac{\partial G^i}{\partial y^i}&=u\left\{(n+1)W-zW_z+2sU-szU_z+(r^2-s^2)U_s\right\},
\end{align}
and consequently, from \eqref{eq:partialG^0} and \eqref{eq:partialG^i}, we have
\begin{align}\label{eq:partialGA}
	\frac{\partial G^A}{\partial y^A}=u\left\{(n+2)W+3sU+A_z + (r^2-s^2)U_s\right\}.
\end{align}
Using \eqref{eq:partialGA}, \eqref{eq:G^0} and \eqref{eq:G^i}, we obtain the result.
\end{proof}
To obtain the Douglas curvature of the cylindrically symmetric Finsler metric \eqref{def:F}, for any differentiable function $ \Theta=\Theta(s,z),$ we adopt the notation $\Psi(\Theta)=-s\Theta_s-z\Theta_z$, and we observe that, for any $m\in\mathbb{Z}^\ast$ we have
\begin{align}
\Psi(\Psi(\Theta))&=-\Psi(\Theta)-s\Psi(\Theta_s)-z\Psi(\Theta_z),\nonumber\\
\frac{\Psi\left(z^m\Theta\right)}{z^m}&=\Psi(\Theta)-m\Theta,\nonumber\\
\frac{\Psi\left(z^m\Theta\right)}{z^m}&=\frac{\Psi\left(z^{m-1}\Theta\right)}{z^{m-1}}-\Theta,\nonumber\\
\Psi\left(z^2\Psi\left(\dfrac{\Theta}{z^2}\right)\right)&=-sz\Psi\left(\dfrac{\Theta_s}{z}\right)-z^2\Psi\left(\dfrac{\Theta_z}{z}\right),\label{psieq1}\\
\dfrac{1}{z}\Psi\left(z^2\Psi\left(\dfrac{\Theta}{z}\right)\right)&=-s\Psi(\Theta_s)-z\Psi(\Theta_z)-z\Psi\left(\dfrac{\Theta}{z}\right),\nonumber\\
\Psi\left(\Theta_z\right)&=\Psi_z(\Theta) + \Theta_z,\label{psieq4}\\
z\Psi_z(\Theta)&=\Psi\left(z\Theta_z\right),\label{eq:psiztheta}\\
\Psi_s(\Theta)&=\Psi(\Theta_s)-\Theta_s,\nonumber\\
 z\Psi_s\left(\frac{\Theta}{z}\right) &= \Psi (\Theta_s),\label{psieq2}\\
\left(z\Psi\left(\frac{\Theta}{z}\right)\right)_z &= \Psi (\Theta_z),\label{psieq3}\\
\Psi\left(z^2\Psi\left(\frac{\Theta_s}{z}\right)\right)&=z\Psi_s\left(z^2\Psi\left(\frac{\Theta}{z^2}\right)\right),\nonumber\\
\Psi_s\left(z^2\Psi\left(\frac{\Theta}{z^2}\right)\right)&=\Psi\left(z\Psi\left(\dfrac{\Theta_s}{z}\right)\right)-z\Psi\left(\dfrac{\Theta_s}{z}\right).\nonumber
\end{align}
Whit this,  
\begin{align}
	\frac{\partial\Theta}{\partial y^0}&= \frac{\Theta_z}{u},\label{partial0}
	\\
	u\frac{\partial\Theta}{\partial y^l}&=\Theta_sx^l +\Psi(\Theta)u_l,
	\label{thetal}\\
 u\frac{\partial }{\partial y^k } \left(\Theta u_ l\right)&=\Theta\delta_{kl} + \Theta_sx^ku_l+ \frac{1}{z}\Psi\left(z\Theta\right)u_ku_l,\label{eq:Thetaul}\\
     u\frac{\partial }{\partial y^j}\left(\Theta u_ku_l\right)&=\Theta\left(\delta_{jk}u_l \right)_{\overrightarrow{kl}} + \Theta_sx^ju_ku_l + \frac{1}{z^2}\Psi\left(z^2\Theta\right)u_ju_ku_l,\label{eq:Thetaukul}\\
 u\frac{\partial }{\partial y^j}\left(\Theta u_ku_lu_i\right)&=\Theta (\delta_{jk}u_lu_i )_{\overrightarrow{kli}} + \Theta_sx^ju_ku_lu_i + \frac{1}{z^3}\Psi\left(z^3\Theta\right)u_ju_ku_lu_i, \label{eq:tetajkli}
\end{align}
where $(.)_{jkl}$ denotes the cyclic permutation (ex.: $(\delta_{jk}u_lu_i )_{\overrightarrow{kli}}=\delta_{jk}u_lu_i+\delta_{jl}u_iu_k + \delta_{ji}u_ku_l$).

\begin{theorem}\label{Dcurv}\normalfont Let $F=\vert\overline{y}\vert{\phi(x^0,r,s,z)}$ be a  Finsler metric defined on $ M $, where $ z=\frac{y^0}{\vert\overline{y}\vert}, $ $r=\vert\overline{x}\vert$, $ s=\frac{\langle\overline{x},\overline{y}\rangle}{\vert\overline{y}\vert} $ and  $TM $  with coordinates   \eqref{coordx}, \eqref{coordy}.  Then the
Douglas curvature of $F$ is given by
\begin{align*}
	D^0_{000}&=\frac{1}{u}R_{zzz},\\
	D^0_{00l}&=\frac{1}{u}\left\{R_{szz}x^l+\Psi(R_{zz})u_l\right\},\\
	D^{0}_{0kl}&=\frac{1}{u}\left\{R_{ssz}x^kx^l + \Psi\left(R_{sz}\right)(x^lu_k)_{\overrightarrow{lk}} + z\Psi\left(\frac{R_z}{z}\right)\delta_{kl} + \frac{1}{z}\Psi\left(z^2\Psi\left(\frac{R_z}{z}\right)\right)u_ku_l \right\},\\
D^0_{jkl}&=\frac{1}{u}\left\{\frac{R_{sss}}{3}x^jx^kx^l+\Psi(R_{ss})x^jx^ku_l +z\Psi\left(\frac{R_s}{z}\right)x^j\delta_{kl}+\Psi\left(z^2\Psi\left(\frac{R}{z^2}\right)\right)u_j\delta_{kl}  \right.\\
    &\qquad\left.+\frac{1}{z}\Psi\left(z^2\Psi\left(\frac{R_s}{z}\right)\right)x^ju_ku_l+ \frac{1}{3z^2}\Psi\left(z^2\Psi\left(z^2\Psi\left(\frac{R}{z^2}\right)\right)\right)u_ju_ku_l\right\}_{\overrightarrow{jkl}},\\
D^i_{000}&=\frac{1}{u}\left\{U_{zzz}x^i - T_{zzz}u_i\right\},\\
D^i_{00l}&=\frac{1}{u}\left\{U_{szz}x^lx^i + \Psi\left(U_{zz}\right)x^iu_l - T_{zz}\delta_{il} - T_{szz}x^lu_ i - \Psi_z\left(T_{z}\right)u_lu_i\right\},\\
D^i_{0kl}&= \frac{1}{u}\left\{U_{ssz}x^kx^lx^i+z\Psi\left(\frac{U_z}{z}\right)\delta_{kl}x^i+ \frac{1}{z}\Psi\left(z^2\Psi\left(\frac{U_z}{z}\right)\right)u_ku_lx^i\right.\\
&\left.\quad\qquad -T_{ssz}x^kx^lu_i - \frac{1}{z^2}\Psi\left(z^2\Psi\left(T_z\right)\right)u_k u_lu_i \right\}\\
    &\quad +\frac{1}{u}\left\{\Psi(U_{sz})u_kx^lx^i  - T_{sz}x^k\delta_{li} -\frac{1}{z}\Psi\left(zT_{sz}\right)x^lu_ku_i \right\}_{\overrightarrow{kl}}+\frac{1}{u}\Psi(T_z)(\delta_{il}u_k)_{\overrightarrow{ikl}},\\
	D^i_{jkl}&= \frac{1}{u}\left\{U_{sss}x^jx^kx^l+\frac{1}{z^2}\Psi\left(z^2\Psi\left(z^2\Psi\left(\frac{U}{z^2}\right)\right)\right)u_ju_ku_l\right.\\
 &\quad+\left. \left[\Psi(U_{ss})u_jx^kx^l+z\Psi\left(\frac{U_s}{z}\right)\delta_{jk}x^l+\Psi_s\left(z^2\Psi\left(\frac{U}{z^2}\right)\right)u_ju_kx^l  
 +\Psi\left(z^2\Psi\left(\frac{U}{z^2}\right)\right) \delta_{jk}u_l\right]_{\overrightarrow{jkl}}\right\}x^i
 \\
&\quad -\frac{1}{u}\left\{\left[T_{ss}\delta_{ij}x^kx^l+\Psi_s(T_s)u_iu_jx^kx^l + \frac{1}{z^2}\Psi\left(z^2\Psi\left(T_s\right)\right)x^ju_ku_lu_i\right]_{\overrightarrow{jkl}}\right.\\
&\quad +z\Psi\left(\frac{T}{z}\right)(\delta_{ji}\delta_{kl})_{\overrightarrow{ikl}}+\Psi\left(T_s\right)\left(x^j(u_i\delta_{kl})_{\overrightarrow{ikl}}+x^k(u_j\delta_{il})_{\overrightarrow{ijl}}+x^l(u_i\delta_{jk})_{\overrightarrow{ijk}}\right)\\
&\quad \left.+\frac{1}{z}\Psi\left(z^2\Psi\left(\frac{T}{z}\right)\right)(\delta_{ji}u_ku_l + \delta_{ik}u_lu_j)_{\overrightarrow{ikl}}+T_{sss}x^jx^kx^lu_i +\frac{1}{z^3}\Psi\left(z^2\Psi\left(z^2\Psi\left(\frac{T}{z}\right)\right)\right)u_ju_ku_lu_i\right\},
\end{align*}
 where, $\Psi(\Theta)=-s\Theta_s-z\Theta_z $, $ u=\vert\overline{y}\vert,$  $ u_i=\frac{\partial u}{\partial y^i}=u^i,$ and $(\cdot)_{\overrightarrow{jkl}}$ denotes cyclic permutation.

\end{theorem}

\begin{proof}
 By Lemma \ref{lem1} and from \eqref{partial0}, \eqref{thetal}, \eqref{eq:Thetaul}, \eqref{eq:Thetaukul}, we have,
\begin{align*}
    D_0{}^0{}_{00}=&\frac{\partial^3}{\partial y^0\partial y^0\partial y^0}(u^2R)=\frac{\partial^2}{\partial y^0\partial y^0}(u^2\frac{R_z}{u})=\frac{\partial}{\partial y^0}(R_{zz}) = \frac{R_{zzz}}{u}\\
    D_0{}^0{}_{0l}=&\frac{\partial^3}{\partial y^0\partial y^0\partial y^l}(u^2R)=\frac{\partial}{\partial y^l}(R_{zz})=\frac{1}{u}\left[R_{szz}x^l + \Psi(R_{zz})u_l\right]
\end{align*}
and using the identity $\Psi(R_{sz}) = z\Psi_s\left(\frac{R_z}{z}\right)$, where the sub index $s$ represents the partial derivative in $s$, we obtain
\begin{align*}
    D_0{}^0{}_{kl}&=\frac{\partial^2}{\partial y^k\partial y^l}\left(uR_z\right)=y^0\frac{\partial^2}{\partial y^k\partial y^l}\left(\frac{R_z}{z}\right)
    = \frac{\partial }{\partial y^k}\left({R_{sz}}x^l +{z} \Psi\left(\frac{R_z}{z}\right)u_l\right)\\
    &=\frac{1}{u}\left\{R_{ssz}x^kx^l + \Psi\left(R_{sz}\right)(x^lu_k)_{\overrightarrow{lk}} + z\Psi\left(\frac{R_z}{z}\right)\delta_{kl} + \frac{1}{z}\Psi\left(z^2\Psi\left(\frac{R_z}{z}\right)\right)u_ku_l \right\}\\
    D_j{}^0{}_{kl}&=(y^0)^2\frac{\partial^3}{\partial y^j\partial y^k\partial y^l}\left(\frac{R}{z^2}\right)=y^0\frac{\partial^2}{\partial y^j\partial y^k}\left(\frac{R_s}{z}x^l + z\Psi\left(\frac{R}{z^2}\right)u_l\right)\\
    &=\frac{\partial }{\partial y ^j}\left[R_{ss}x^kx^l + z\Psi\left(\frac{R_s}{z}\right)(x^lu_k)_{\overrightarrow{kl}} + z^2\Psi\left(\frac{R}{z^2}\right)\delta_{kl} + \Psi\left(z^2\Psi\left(\frac{R}{z^2}\right)\right)u_ku_l\right].
\end{align*}
From \eqref{eq:G^i},  \eqref{thetal}, \eqref{eq:Thetaul} and \eqref{eq:psiztheta}, $D_0{}^i{}_{00}$ and $D_0{}^i{}_{0l}$ are directly obtained. Using the properties of $\Psi $ we have,
\begin{align*}
    D_0{}^i{}_{kl}&=y^0\frac{\partial^2}{\partial y^k\partial y^l}\left(\frac{U_z}{z}x^i - \frac{T}{z}u_i\right)\\
    &=\frac{\partial }{\partial y^k} \left(U_{sz}x^lx^i+z\Psi\left(\frac{U_z}{z}\right)u_lx^i - T_z\delta_{li} - T_{sz}x^lu_i - \Psi(T_z)u_lu_i\right).
\end{align*}
Analogous to the previous cases, using \eqref{eq:tetajkli}, we have
\begin{align*}
D_j{}^i{}_{kl}&=\frac{\partial^2}{\partial y^j\partial y^k}\left(2Uuu_lx^i+uU_sx^lx^i+u\Psi(U)u_lx^i-2uTu_lu_i-uT\delta_{li}-uT_sx^lu_i-\dfrac{\Psi(zT)}{z}uu_lu_i\right)\\
    &=\frac{\partial}{\partial y^j}\left\{\Psi\left(z^2\Psi\left(\dfrac{U}{z^2}\right) \right)u_ku_lx^i+z^2\Psi\left(\dfrac{U}{z^2}\right)\delta_{kl}x^i+z\Psi\left(\dfrac{U_s}{z}\right)(x^ku_l)_{ \overrightarrow{kl}}x^i\right.\\
    &\quad \left.-\dfrac{1}{z}\Psi\left(z^2\Psi\left(\dfrac{T}{z}\right)\right) u_ku_lu_i-z\Psi\left(\dfrac{T}{z}\right)(\delta_{kl}u_i)_{ \overrightarrow{kli}}-\Psi(T_s)(x^ku_l)_{ \overrightarrow{kl}}u_i+U_{ss}x^lx^kx^i\right.\\
    &\quad\left.-T_{ss}u_ix^kx^l-T_s(x^k\delta_{li})_{ \overrightarrow{kl}}\dfrac{}{}\right\rbrace.\end{align*}

\end{proof}

\begin{theorem}\label{maintheo1}\normalfont
    Let $F=\vert u\vert \phi(x^0,r,s,z),$ be a Finsler metric defined on $I\times \mathbb{B}^n(\rho)$, $n\geq 3$, where $z=\frac{y^0}{u},$ $r=\vert \overline{x}\vert$ and $s=\frac{\langle \overline{x},\overline{y}\rangle}{u}$, and $TM$ defined with coordinates \eqref{coordx}, \eqref{coordy}. Then $F$ has vanishing Douglas curvature if, and only if, $\phi$ satisfies
    \begin{align}
        (a)\quad z\Psi\left(\frac{U_s}{z}\right)&=0 &(b)\quad z\Psi\left(\frac{U_z}{z}\right)&=0, &(c)\quad U_{zzz}&=0,\label{eq:pdeU}\\
        (a)\quad z\Psi\left(\frac{R_s}{z}\right)&=0 & (b)\quad z\Psi\left(\frac{R_z}{z}\right)&=0, &(c)\quad R_{zzz}&=0,\label{eq:pdeR}\\
        (a)\quad z\Psi\left(\frac{T}{z}\right)&=0, &(b)\quad T_{zz}&=0,\label{eq:pdeT}
        \end{align}
    where $U, R$ and $T$ are given in \eqref{def:U}, \eqref{def:R} and \eqref{def:T}, respectively.
\end{theorem}

\begin{proof} Suppose $F$ has vanishing Douglas curvature. Consider the orthonormal matrix $O\in O(n)$ (See the proof of Proposition 1.3.1 in \cite{GuoMo2018Book} or the proof of Lemma 1 in \cite{Solorzano2022}) such that 
   \begin{align*}
       \tilde{x}&=O\overline{x}=\left(\vert \overline{x}\vert, 0,\ldots,0\right)\\
       \Tilde{y}&=O\overline{y}=\left(\frac{\langle \overline{x}, \overline{y}\rangle}{\vert \overline{x}\vert}, \frac{\sqrt{\vert \overline{x}\vert^2\vert \overline{y}\vert^2 - \langle \overline{x},\overline{y}\rangle^2}}{\vert \overline{x}\vert},0,\ldots,0\right).
   \end{align*} 
   For the invariance of $r,s$ and $z$ under the action $O,$ from $D_0{}^0{}_{00}=0$, we obtain $R_{zzz}=0$. From $D_0{}^0{}_{33}=0$, we get \begin{equation}\label{Deq1}
z\Psi\left(\dfrac{R_z}{z}\right)=0.
   \end{equation}
   Using property \eqref{psieq1} and \eqref{Deq1}, we have
\begin{equation}\label{Deq2}
       \Psi\left(z^2\Psi\left(\dfrac{R}{z^2}\right)\right)=-sz\Psi\left(\dfrac{R_s}{z}\right).
   \end{equation}
   From $D_1{}^0{}_{33}=0$, we obtain
   \begin{equation}\label{Deq3}    rz\Psi\left(\dfrac{R_s}{z}\right)+\dfrac{s}{r}\Psi\left(z^2\Psi\left(\dfrac{R}{z^2}\right)\right)=0.
   \end{equation}
   Substituting \eqref{Deq2} into \eqref{Deq3}, we get
   \[\left(\dfrac{r^2-s^2}{r}\right)\left(z\Psi\left(\dfrac{R_s}{z}\right)\right)=0.\]Hence, \[z\Psi\left(\dfrac{R_s}{z}\right)=0.
\]Thus, \eqref{eq:pdeR} is satisfied.

From $D_0{}^1{}_{00}=0$, we get $U_{zzz}=0$. From $D_0{}^1{}_{33}=0$, we have 
\[z\Psi\left(\dfrac{U_z}{z}\right)=0.\]From $D_3{}^1{}_{31}=0$, we obtain \[rz\Psi\left(\dfrac{R_s}{z}\right)+\dfrac{s}{r}\Psi\left(z^2\Psi\left(\dfrac{R}{z^2}\right)\right)=0.\]Similarly, as in the case of $R$, we conclude
   \[\left(\dfrac{r^2-s^2}{r}\right)\left(z\Psi\left(\dfrac{U_s}{z}\right)\right)=0.\]Therefore, \[z\Psi\left(\dfrac{U_s}{z}\right)=0,
\]and thus \eqref{eq:pdeU} is satisfied.

From $D_0{}^3{}_{03}=0$ and $D_3{}^3{}_{33}=0$, 
 we have $T_{zz}=0$ and $z\Psi\left(\dfrac{T}{z}\right)=0$.

Conversely, assume that $\phi$ satisfies \eqref{eq:pdeU}, \eqref{eq:pdeR} and \eqref{eq:pdeT}. From \eqref{eq:pdeR} $(a)$, we get $D_0{}^0{}_{00}=0$. Using property \eqref{psieq2} and \eqref{eq:pdeR} $(a),\;(b)$, we obtain
\begin{equation}\label{Deq4} \Psi(R_{ss})=z\Psi_s\left(\dfrac{R_s}{z}\right)=0,
\end{equation}
and 
\begin{equation}\label{Deq5} \Psi(R_{sz})=z\Psi_s\left(\dfrac{R_z}{z}\right)=0.
\end{equation}
By property \eqref{psieq3} and \eqref{eq:pdeR} $(b)$, we have
\begin{equation}\label{Deq6}
\Psi(R_{zz})=\left(z\Psi\left(\dfrac{R_z}{z}\right)\right)_z=0.
\end{equation}
From \eqref{Deq6} and \eqref{eq:pdeR} $(c)$, we get
\begin{equation}\label{Deq7}
    R_{zzs}=0.
\end{equation}
Therefore, by \eqref{Deq6} and \eqref{Deq7}, we obtain $D_0{}^0{}_{0l}=0$. From \eqref{Deq5} and \eqref{Deq7}, we also have
\begin{equation}\label{Deq8}
 R_{zss}=0.   
\end{equation}
Consequently, by \eqref{Deq8}, \eqref{Deq4} and \eqref{eq:pdeR} $(b)$, we have $D_0{}^0{}_{kl}=0$. From \eqref{Deq4} and \eqref{Deq8}, we obtain
\begin{equation}\label{Deq9}
    R_{sss}=0.
\end{equation}
Also, by property \eqref{psieq1} and \eqref{eq:pdeR} $(a),\;(b)$, we have
\begin{equation}\label{Deq10}
\Psi\left(z^2\Psi\left(\dfrac{R}{z^2}\right)\right)=-sz\Psi\left(\dfrac{R_s}{z}\right)-z^2\Psi\left(\dfrac{R_z}{z}\right)=0. 
\end{equation}
Thus, by \eqref{Deq9}, \eqref{Deq4}, \eqref{eq:pdeR} $(a)$ and \eqref{Deq10}, we conclude that $D_j{}^0{}_{kl}=0$. From \eqref{eq:pdeR} $(c)$ and \eqref{eq:pdeT} $(b)$, we have $D_0{}^i{}_{00}=0$. Now, by \eqref{eq:pdeU}, analogously as in the case of $R$ we obtain
\begin{align}
\Psi(U_{ss})=\Psi(U_{sz})=\Psi(U_{zz})=0\label{Deq11},\\
U_{zss}=U_{zss}=U_{sss}=0,\quad \Psi\left(z^2\Psi\left(\dfrac{U}{z^2}\right)\right)=0.\label{Deq12}
\end{align}
On the other hand, by property \eqref{psieq4} and \eqref{eq:pdeT} $(b)$, we have
\begin{equation}\label{Deq13}
\Psi_z(T_z)=\Psi(T_{zz})-T_{zz}=0.    
\end{equation}
Therefore, by \eqref{Deq11}, \eqref{Deq12}, \eqref{eq:pdeT} $(b)$ and \eqref{Deq13}, we obtain $D_0{}^i{}_{0l}=0$. By property \eqref{psieq3} and \eqref{eq:pdeT} $(a)$, we get
\begin{equation}\label{Deq14}
\Psi(T_z)=\left(z\Psi\left(\dfrac{T}{z}\right)\right)_z=0.
\end{equation}From \eqref{Deq14} and \eqref{eq:pdeT} $(b)$, we obtain 
\begin{equation}\label{Deq15}
    T_{sz}=0.
\end{equation}
Consequently, by \eqref{Deq11}, \eqref{Deq12}, \eqref{eq:pdeU} $(b)$, \eqref{Deq14} and \eqref{Deq15}, we get $D_0{}^i{}_{kl}=0$. By property \eqref{psieq2} and \eqref{eq:pdeT} $(a)$, we obtain
\begin{equation}\label{Deq16} \Psi(T_s)=z\Psi_s\left(\dfrac{T}{z}\right)=0.
\end{equation} From \eqref{Deq16} and \eqref{Deq15}, we have
\begin{equation}\label{Deq17}
    T_{ss}=0.
\end{equation}
Therefore, by \eqref{Deq11}, \eqref{Deq12}, \eqref{eq:pdeU}, \eqref{eq:pdeT} $(a)$, \eqref{Deq16} and \eqref{Deq17}, we obtain $D_j{}^i{}_{kl}=0$.
\end{proof}

\begin{corollary}\label{cor:PDE1}\normalfont
 Let $F= u \phi(x^0,z,r,s),$ be a Finsler metric defined on $I\times \mathbb{B}^n(\rho)$, $n\geq 3$, where $r=\vert \overline{x}\vert$, $s=\frac{\langle \overline{x},\overline{y}\rangle}{u}$ and $z=\frac{y^0}{u},$  and $TM$ defined with coordinates \eqref{coordx}, \eqref{coordy}. Suppose that $F=F(x,y)$ has vanishing Douglas curvature. Then, there exist some differentiable functions $f_i=f_i(x^0,r), g_i=g_i(x^0,r) $ and $h_i=h_i(x^0,r),$ such that $\phi$ satisfies,
 \begin{align}\label{eq:reduced}
    z\psi_{x^0}+\frac{s}{r}\psi_r + \left[1-2(r^2-s^2)U\right]\psi_s - 2L\psi_z=0,
\end{align}
where $\psi=\sqrt{r^2-s^2}\Omega,$ and,
\begin{align*}
U&=f_1\frac{s^2}{2}+f_2 sz + f_3\frac{z^2}{2} + f_4, \\
L&= g_1\frac{s^2}{2} +g_2sz +g_3\frac{z^2}{2}+g_4 +z(h_1s+h_2z)-sz(f_1\frac{s^2}{2}+f_2 sz + f_3\frac{z^2}{2} + f_4).
\end{align*}
\end{corollary}
\begin{proof}
    From \eqref{eq:pdeU}, \eqref{eq:pdeR} and \eqref{eq:pdeT} we have that there are differentiable functions $f_i=f_i(x^0,r), g_i=g_i(x^0,r)$, $h_i=h_i(x^0,r)$ such that
    \begin{align*}
        U&=f_1\frac{s^2}{2}+f_2 sz + f_3\frac{z^2}{2} + f_4,\\
        R&=g_1\frac{s^2}{2}+g_2 sz + g_3\frac{z^2}{2} + g_4,\\
        T&=h_1s + h_2z.
    \end{align*}
From  \eqref{def:R} and \eqref{def:T} we have that $R+zT=A+szU$, and then,
\begin{align*}
L&=R+zT-szU\\
 &= g_1\frac{s^2}{2} +g_2sz +g_3\frac{z^2}{2}+g_4 +z(h_1s+h_2z)-sz(f_1\frac{s^2}{2}+f_2 sz + f_3\frac{z^2}{2} + f_4).
\end{align*}
From definition of $U$ and $A$ in \eqref{def:U} and \eqref{def:A}, we have
\begin{align}
    \phi_{zz}p_1-\phi_{sz}p_2&=2\Lambda U, \label{eq:UA01}\\
    -(r^2-s^2)\phi_{sz}p_1+(\Omega +(r^2-s^2)\phi_{ss})p_2&=2\Lambda L, \label{eq:UA02}
\end{align}
where
\begin{align}
   p_1&:=\left(\varphi_s-\frac{2}{r}\phi_r\right)=z\phi_{x^0s} - \frac{1}{r}\phi_r + \frac{s}{r}\phi_{rs}+\phi_{ss},\label{def:p1}\\
   p_2&:=\left(\varphi_z-2\phi_{x^0}\right)=z\phi_{x^0z} - \phi_{x^0} + \frac{s}{r}\phi_{rz} + \phi_{sz},\label{def:p2}
\end{align}
and $\Omega, \Lambda$ are given in \eqref{DefOmega} and \eqref{Def:Lambda} respectively. Due to the fact $\Lambda\neq 0, $ the system \eqref{eq:UA01}- \eqref{eq:UA02} is equivalent to
\begin{align}
    p_1&=2\left[\left(\Omega+(r^2-s^2)\phi_{ss}\right)U +\phi_{sz}L\right],\label{eq:p1AU}\\
    p_2&=2\left[\phi_{zz}L+(r^2-s^2)\phi_{sz}U\right].\label{eq:p2AU}
\end{align}
From the definition of $p_1$ and $p_2$ in \eqref{def:p1} and \eqref{def:p2}, we have, $sp_1+zp_2=-z\Omega_{x^0} - \frac{s}{r}\Omega_r-\Omega_s.$ And using \eqref{eq:UA01}-\eqref{eq:UA02}, we obtain,

\begin{align*}
    2sU\Omega + z\Omega_{x^0}+\frac{s}{r}\Omega_r + \left[1-2(r^2-s^2)U\right]\Omega_s - 2L\Omega_z=0
\end{align*}
which is equivalent to \eqref{eq:reduced}, using the substitution $\psi=\sqrt{r^2-s^2}\Omega.$
\end{proof}

\begin{remark}\normalfont
    From, \eqref{eq:p1AU}, \eqref{eq:p2AU} and due to the fact a cylindrically symmetric Finsler metric $F=\vert \overline{y}\vert \phi(x^0,r,s,z)$ is projectively flat (See Theorem 1.1 in \cite{Liu2024}) if, and only if, $p_1=p_2=0,$ then $F$ is protectively flat if, and only if, $L=U=0.$ 
\end{remark}

\section{Douglas metric examples}
Using Theorem \ref{maintheo1} and Corollary \ref{cor:PDE1} we obtain the next cylindrically symmetric Douglas metrics,
\begin{example}\normalfont
Let $\phi(x^0,r,s,z)$ be a function defined by
\begin{align}
    \phi(x^0,r,s,z)=\sqrt{1+r^2-s^2+e^{x^0}z^2} + sh(r),
\end{align}
where $h(r)$ is any function such that $\phi$ is positive. For instance, consider $h(r)=\dfrac{k}{1+r^2}$ where $|k|<2$. With this, 
\begin{align*}
    U&=-\frac{r^2-s^2+1}{1+r^2},\\
    R&=\frac{1}{4}\frac{(nr^2z-4ns+nz+4s)z}{(n+2)(1+r^2)},\\
    T&=\frac{1}{2}\frac{r^2z-6s+z}{(n+2)(1+r^2)}.
\end{align*} Then, by Theorem \ref{maintheo1}, the following Finsler metric on $\mathbb{R}\times \mathbb{B}^n(\rho)$
\begin{align*}
    F(x,y)=\sqrt{(1+\vert \overline{x}\vert^2)\vert \overline{y}\vert^2 -\langle \overline{x},\overline{y}\rangle^2 +e^{x^0}(y^0)^2 } +\frac{1}{1+\vert \overline{x}\vert^2}\langle \overline{x},\overline{y}\rangle,
\end{align*}
 is a cylindrically symmetric Douglas metric. 
\end{example}

\begin{example}\normalfont
Let $\phi(x^0,r,s,z)$ be a function defined by
\begin{align}
    \phi(x^0,r,s,z)=\sqrt{1+r^2+s^2+e^{x^0}z^2} + sh(r),
\end{align}
where $h(r)$ is any function such that $\phi$ is positive. For instance, consider $h(r)=\dfrac{k}{1+r^2}$ where $|k|<2$. With this, 
\begin{align*}
    U&=-\frac{-s^2}{1+3r^2 + 2r^4},\\
    R&=\frac{1}{4}\frac{(2nr^4z-8nr^2s+3nr^2z-4ns+nz-4s+4s)z}{(n+2)(1+3r^2+2r^4)},\\
    T&=\frac{1}{2}\frac{2r^4z-8r^2s+3r^2z-2s+z}{(n+2)(1+3r^2 + 2r^4)}.
\end{align*} Then, by Theorem \ref{maintheo1}, the following Finsler metric on $\mathbb{R}\times \mathbb{B}^n(\rho)$
\begin{align*}
    F(x,y)=\sqrt{(1+\vert \overline{x}\vert^2)\vert \overline{y}\vert^2 +\langle \overline{x},\overline{y}\rangle^2 +e^{x^0}(y^0)^2 } +h(\vert \overline{x}\vert)\langle \overline{x},\overline{y}\rangle,
\end{align*}
where $h(\vert \overline{x}\vert)<1$, is a cylindrically symmetric Douglas metric. 
\end{example}

\begin{example}\normalfont
Let $\phi(x^0,r,s,z)$ be a function defined by
\begin{align}
    \phi(x^0,r,s,z)=\frac{\sqrt{h(x^0)^2g(r)^2z^2+1}}{g(r)} + h(x^0)z,
\end{align}
where $h(x^0)>0$ is any function such that $\phi$ is positive.  With this, 
\begin{align*}
    U&=\frac{1}{2r}\frac{g'(r)}{g(r)},\\
    R&=\frac{1}{2(n+2)}\frac{n(g(r)h'(x^0)rz + 2h(x^0)g'(r)s)z}{rg(r)h(x^0)},\\
    T&=\frac{g(r)h'(x^0)rz + 2h(x^0)g'(r)s}{(n+2)rg(r)h(x^0)}.
\end{align*} Then, by Theorem \ref{maintheo1}, the following Finsler metric on $\mathbb{R}\times \mathbb{B}^n(\rho)$
\begin{align*}
    F(x,y)=\frac{\sqrt{\vert\overline{y}\vert^2 + h(x^0)^2 g(\vert\overline{x}\vert)^2(y^0)^2}}{g(\vert\overline{x}\vert)} +h(x^0)y^0,
\end{align*}
is a cylindrically symmetric Douglas metric. 
\end{example}

\begin{example}\cite{Solorzano2022}
    Considering $U=\frac{1}{2}\frac{g'(r)}{rg(r)}$ and $L=\frac{1}{2}\frac{szg'(r)}{rg(r)}.$ The function $\psi=G\left(\frac{r^2-s^2}{g(r)^2}, zg(r)\right)$ solves the equation \eqref{eq:reduced}. Then, if $ G=\frac{\sqrt{r^2-s^2}}{g(r)\sqrt{g(r)^2z^2+1}}$, the PDE \[\Omega=-s^2\left[\frac{\phi}{s}\right]_s - sz\left[\frac{\phi}{s}\right]_z=\frac{1}{g(r)\sqrt{g(r)^2 z^2+1}}\]
    give us
    \[\phi=\frac{\sqrt{g(r)^2z^2+1}}{g(r)} + h(x^0,r,\frac{z}{s})s.\]
    If $h(x,r,\frac{z}{s}) = h(x^0)\frac{z}{s}$, we obtain the next cylindrically symmetric Douglas metric, 
    \[F(x,y)=h(x^0)y^0 + \frac{\sqrt{g(\vert\overline{x}\vert)^2(y^0)^2 + 1}}{g(\vert\overline{x}\vert)}, \]
    where $\vert h(x^0)\vert<1.$
\end{example}

\begin{example}\label{ex:d1}
    Similarly to the previous example, considering $U=\frac{1}{2}\frac{g'(r)}{rg(r)}$, $L=\frac{1}{2}\frac{szg'(r)}{rg(r)}$ and $G=\frac{\sqrt{r^2-s^2}}{g(r)}\left(1+\frac{1}{(g(r)^2z^2+1)^{3/2}}\right)$, we get
    \[\phi =h(x^0,r,\frac{z}{s})s + \frac{1}{g(r)}\left(1+\frac{2g(r)^2z^2 + 1}{\sqrt{g(r)^2z^2 + 1}}\right). \]
    If $h(x^0,r,\frac{z}{s})s=h(x^0)z,$ we obtain the next cylindrically symmetric Douglas metric,
    \[F(x,y)=h(x^0)y^0 + \frac{1}{g(\vert \overline{x}\vert)}\left( \vert \overline{y}\vert + \sqrt{g(\vert \overline{x}\vert)^2(y^0)^2 + \vert \overline{y}\vert^2}+\frac{g(\vert \overline{x}\vert)^2(y^0)^2}{\sqrt{g(\vert \overline{x}\vert)^2(y^0)^2 + \vert \overline{y}\vert^2}}\right),\]
where     $\vert h(x^0)\vert<1$ and $g(r)>0.$
\end{example}

\begin{example}
    Let $\vert h(x^0)\vert<1, $ $g(r)>0$ and $f(x^0)>0$ differentiable functions. Motivated by the Example \ref{ex:d1}, the next cylindrically symmetric Finsler metric
    \[F(x,y)=h(x^0)y^0 + \frac{1}{g(\vert \overline{x}\vert)}\left( \vert \overline{y}\vert +\frac{2g(\vert \overline{x}\vert)^2(y^0)^2+f(x^0)\vert \overline{y}\vert^2}{\sqrt{g(\vert \overline{x}\vert)^2(y^0)^2 + f(x^0)\vert \overline{y}\vert^2}}\right),\]
     has vanishing Douglas curvature with,
    \begin{align*}
        U&=\frac{1}{2}\frac{g'(r)}{rg(r)}\\
        R&=-\frac{1}{12}\frac{4ng(r)^2f'(x^0)rz^2 -12nf(x^0)g(r)g'(r)sz+(n+2)f(x^0)f'(x^0)r}{(n+2)rg(r)^2f(x^0)}\\
        T&=-\frac{2}{3}\frac{g(r)f'(x^0)rz -3f(x^0)g'(r)s}{(n+2)rg(r)f(x^0)}.
    \end{align*}
\end{example}

\end{document}